\documentclass[a4paper,reqno]{amsart}
 \usepackage{amssymb,amsmath,amscd,graphicx,color,epstopdf,mathtools}
\usepackage[all]{xy}
\usepackage[all]{xy}
\usepackage{color}
\usepackage{hyperref}
\usepackage{enumerate}
\usepackage{amsthm}

\usepackage{epsfig}
\usepackage[english]{babel}
 
\let\oldmarginpar\marginpar
\renewcommand\marginpar[1]
{\oldmarginpar{\tiny\bf \begin{flushleft} #1 \end{flushleft}}}

   \renewcommand{\a}{\alpha}

  \newcommand{\w}{\omega}
  
   \newcommand{\R}{\mathbb{R}}
   \newcommand{\C}{\mathbb{C}}

    \newcommand{\cL}{\mathcal{L}}

\newcommand{\std}{\mathrm{std}}
\renewcommand{\gg}{\mathfrak{g}}        
\renewcommand{\aa}{\mathfrak{a}} 
\newcommand{\pp}{\mathfrak{p}}          

\newtheorem{theorem}{Theorem}[section] 
\newtheorem{lemma}[theorem]{Lemma}     

\newtheorem{proposition}[theorem]{Proposition}
\newtheorem{remark}[theorem]{Remark}

\begin{document}

\author{David Mart\'inez Torres}

\address{PUC-Rio de Janeiro\\
Departamento de Matem\'atica \\ Rua Marqu\^es de S\~ao Vicente, 225\\
G\'avea - 22451-900, Rio de Janeiro, Brazil }

\email{dfmtorres@gmail.com}

\title{Semisimple coadjoint orbits and cotangent bundles}
\begin{abstract} 
Semisimple (co)adjoint orbits through real hyperbolic elements are well-known to be symplectomorphic to cotangent bundles.
We provide a new proof of this fact based on elementary results on both Lie theory 
and symplectic geometry.  Our proof establishes a new connection between the Iwasawa horospherical projection and the symplectic geometry of real hyperbolic (co)adjoint orbits.
\end{abstract}

\maketitle

\section{Introduction} 

The coadjoint representation has Poisson nature:  the Lie bracket of a Lie algebra $\gg$ canonically induces a linear Poisson
bracket on its dual $\gg^*$. The symplectic leaves of the linear Poisson structure  are the coadjoint orbits.
The induced symplectic structure on a coadjoint orbit is the so called Kostant-Kirillov-Souriau (KKS) symplectic structure.

For any Poisson structure the understanding of the 
symplectic structure of any of it leaves is fundamental; for duals of  Lie algebras this understanding is even more important, 
for it has deep implications on Hamiltonian group actions and representation theory \cite{LG,Ki}.

When $\gg$ is  of compact type, coadjoint orbits --which are the classical complex flag manifolds--
are compact, and therefore more tools are available for the study of their
symplectic geometry \cite{LG}.

Global aspects of the symplectic geometry of non-compact coadjoint orbits are much harder to grasp. The first results --motivated by the 
orbit method-- prove that coadjoint orbits of nilpotent groups and of groups of exponential type are
symplectomorphic to symplectic vector spaces \cite{Ve,Pu,Pe}. For complex semisimple groups, for which the Killing form intertwines
adjoint and coadjoint actions --so one can speak of the KKS symplectic structure of an adjoint orbit-- Lisiecki proved the following:

\begin{theorem}[\cite{Li}]\label{thm:main0} Let $G$ be a connected, complex semisimple Lie group and let $\gg$ denote its Lie algebra. Let $X\in \gg$ be
a semisimple element and let $P\subset G$ be its normalizer. 

There is a canonical symplectomorphism  between the adjoint orbit 
$\mathrm{Ad}(G)_X\subset \gg$
with its holomorphic KKS symplectic structure,
and the cotangent bundle of the  flag manifold $G/P$ with its standard Liouville symplectic structure twisted
by the pull back of a closed holomorphic 2-form on the flag manifold.
\end{theorem}
It should be noted that the twisting form is never zero, so no semisimple coadjoint orbit is a holomorphic cotangent bundle.

Instances of Theorem \ref{thm:main0} had been rediscovered a number of times, mostly in the context of non-compact real semisimple Lie groups, and using
different techniques. 

Arnold \cite{Ar} states Theorem \ref{thm:main0}  
for regular semisimple orbits of $\mathrm{SL}(n+1,\C)$. He also observes that if $\mathrm{SL}(n+1,\C)$ is regarded
as a real Lie  group --so orbits are now regarded as symplectic manifolds with the real part of the holomorphic KKS form-- 
then real hyperbolic orbits  (those having real eigenvalues) are canonically symplectomorphic to cotangent bundles. 

More generally, let $G$ be a connected, non-compact semisimple Lie group with finite center. 
Once an Iwasawa decomposition  $G=KAN$  has been fixed, then real hyperbolic
elements are those conjugated to elements in the closure of  
the fixed positive Weyl chamber $\mathrm{Cl}(\aa^+)\subset \aa$ of the Lie
algebra of $A$. If $H\in \gg$ is real hyperbolic, then \cite{Bi,ABB,GGS} provide three different  proofs of the existence of a symplectomorphism:
\begin{equation}\label{eq:isom}
(\mathrm{Ad}(G)_H,\w_{KKS})\cong (T^*\mathrm{Ad}(K)_H,\w_\std),
\end{equation}
where $\w_\std$ is the standard Liouville symplectic form on the cotangent bundle of the real flag manifold.

%

The symplectomorphism in \ref{eq:isom} is an instance of Theorem \ref{thm:main0}, as it follows by applying an involution on the complex 
Lie algebra to pass to the appropriate real form, and observing that the real flag manifold is a Lagrangian section. However,
from the point of view of symplectic geometry  the case of real hyperbolic orbits is by far the most interesting, as it exhibits
a cotangent bundle with no twisting. 

The proofs of Theorem \ref{thm:main0} and its version for real hyperbolic orbits are far from elementary. 
They rely  on  deep results on classification of complete Lagrangian fibrations with contractible fibers and algebraic geometry of
coadjoint orbits \cite{Li}, integrable systems  \cite{ABB}, 
 hyperKahler reduction\cite{Bi},
or integration of Lie algebra actions \cite{GGS}. 

The purpose of this note is to construct the symplectomorphism in \ref{eq:isom} using just elementary facts on both Lie theory 
and symplectic geometry, thus shedding some new light on the symplectic geometry of semisimple (co)adjoint orbits.

The key ingredients in our strategy are the full use of the \emph{canonical ruling} of the adjoint orbit and the description of new aspects
the \emph{symplectic
geometry of the Iwasawa projections}.
In what follows we briefly discuss the main ideas behind our approach, and compare it with  \cite{Li,ABB,Bi,GGS}:

We assume without loss of generality that $H\in \mathrm{Cl}(\aa^+)$.
The Iwasawa decomposition defines a well-known canonical ruling on the adjoint orbit $\mathrm{Ad}(G)_H$. The ruling, together
with the Killing form,  determine a diffeomorphism
\begin{equation}\label{eq:emb}i\colon T^*\mathrm{Ad}(K)_H\rightarrow \mathrm{Ad}(G)_H
 \end{equation}
 extending the inclusion of real flag manifold $\mathrm{Ad}(K)_H\hookrightarrow\mathrm{Ad}(G)_H$.
 
Of course, the diffeomorphism (\ref{eq:emb})  appears in \cite{Li,ABB,GGS}, but it is not fully exploited. 
\begin{itemize}
 \item In \cite{ABB} non-trivial theory
of complete Lagrangian fibrations is used to construct a symplectomorphism 
$\varphi\colon (T^*\mathrm{Ad}(K)_H,\w_{\mathrm{std}})\rightarrow (\mathrm{Ad}(G)_H,\w_{\mathrm{KKS}})$, with the property 
of being the unique symplectomorphism
which (i) extends the identity on  $\mathrm{Ad}(K)_H$ and (ii) is a morphism  fiber bundles, where $\mathrm{Ad}(G)_H$
has the bundle structure induced by $i$ in (\ref{eq:emb}). 

It can be checked that $\varphi$ as constructed in \cite{ABB} coincides with $i$ in  (\ref{eq:emb});
the uniqueness statement is a consequence of the absence of non-trivial symplectic automorphisms of the cotangent bundle
preserving the zero section and the fiber bundle structure.
\item In \cite{GGS} a complete Hamiltonian action of $\gg$ on 
$(T^*\mathrm{Ad}(K)_H,\w_{\mathrm{std}})$ is built. The momentum map 
$\mu\colon (T^*\mathrm{Ad}(K)_H,\w_{\mathrm{std}})\rightarrow (\mathrm{Ad}(G)_H,\w_{\mathrm{KKS}})$ is the desired
symplectomorphism;  the authors also show that the momentum map $\mu$  matches  $i$ in (\ref{eq:emb}).
\item The first step in  \cite{Li} is  a classification of (holomorphic) complete Lagrangian fibrations with contractible fiber. Next, symplectic induction
is used to produce for each semisimple element $X$ one such fibration, endowed with a Hamiltonian $G$-action of, thus mapping 
into the coadjoint orbit through $X$.
Then algebraic geometric results are used to show that the $G$-action is transitive, so the momentum map is in fact a symplectic biholomorphism. 
\end{itemize}
Summarizing, the proofs in \cite{Li,ABB,GGS} amount to global constructions on  non-compact symplectic manifolds,  something which always presents
technical difficulties (in the holomorphic setting of \cite{Li} the technicalities involve algebraic geometry). \footnote{In \cite{Ki}, Corollary 1, a proof of the isomorphism of a regular coadjoint orbit 
with the cotangent bundle 
is presented,
but it is not correct as the completeness issues are entirely ignored.}

The approach in \cite{Bi} is entirely different: the moduli space of solutions to Nahm's equations is symplectically identified on the
one hand with an adjoint orbit, and on the other hand with the cotangent bundle of its real flag manifold (and the
final symplectomorphisms is by no means explicit).

In this note we will use a simpler and more direct approach. We shall take full advantage of the ruling structure on real hyperbolic orbits 
to prove the equality:
\begin{equation}\label{eq:rul}
 \w_{\mathrm{std}}=i_*\w_{\mathrm{KKS}}.
 \end{equation}
 In fact, this is the approach  sketched by Arnold \cite{Ar}.
 
We shall omit the diffeomorphism $i$ in (\ref{eq:emb}) in the notation whenever there is no risk of confusion, so we consider
$\w_{\mathrm{std}},\w_{\mathrm{KKS}}$ as
two symplectic forms on $\mathrm{Ad}(G)_H$ whose equality we will check using the following strategy:

 Basic symplectic linear algebra \cite{MS} implies that to prove (\ref{eq:rul}) at $x\in \mathrm{Ad}(G)_H$
 it is enough to find $\cL_v,\cL_h\subset T_x\mathrm{Ad}(G)_H$ such that:
 \begin{enumerate}
\item [(i)] $\cL_v,\cL_h$  are  Lagrangian subspaces for $\w_{KKS}$ and $\w_{\mathrm{std}}$ respectively;
 \item[(ii)]  $\cL_v\cap \cL_h=\{0\}$; 
 \item[(iii)] $\w_{\mathrm{std}}(x)(Y,Z)=\w_{\mathrm{KKS}}(x)(Y,Z),\,\, \forall\, Y\in \cL_v,\,Z\in \cL_h$.
\end{enumerate}
As the notation suggests, $\cL_v$ will be the vertical tangent space coming from the fiber bundle structure defined by the ruling. This
vector subspace is trivially Lagrangian for $\w_{\mathrm{std}}$ and easily
seen to be Lagrangian for $\w_{\mathrm{KKS}}$ \cite{ABB}.

Let $\mathrm{Ad}(K)_H^g$ denote the image of $\mathrm{Ad}(K)_H$ by $\mathrm{Ad}(g)$. For any $x\in \mathrm{Ad}(G)_H$ 
transitivity of the adjoint action implies
the existence of $g\in G$ so that $x\in\mathrm{Ad}(K)_H^g.$
The `horizontal' subspace $\cL_h$ at $x$ will be the tangent space to $\mathrm{Ad}(K)_H^g$.
Because the zero section $\mathrm{Ad}(K)_H$ is Lagrangian
w.r.t. $\w_{\mathrm{KKS}}$ \cite{ABB},
 $G$-invariance of $\w_{\mathrm{KKS}}$ implies that $\cL_h$  is 
a Lagrangian subspace w.r.t $\w_{\mathrm{KKS}}$. If $\mathrm{Ad}(K)_H^g$
it is to be Lagrangian w.r.t to $\w_{\mathrm{std}}$, it should correspond to a closed 1-form in
$\mathrm{Ad}(K)_H$. In fact, it will be the graph of an exact 1-form, and the projections associated to the Iwasawa decomposition
will play a crucial role to determine a potential.

The horospherical projection $H\colon G\colon \rightarrow \aa$ is defined by $x\in K\mathrm{exp}H(x)N$. A pair $H\in \aa$, $g\in G$ 
 determines a function:
\[F_{g,H}\colon K\rightarrow \R,\,\, k\mapsto \langle H, H(gk)\rangle.\] Under the assumption $H\in \mathrm{Cl}(\aa^+)$
the function descends to the real flag 
manifold $\mathrm{Ad}(K)_H\cong K/Z_K(H)$, where $Z_K(H)$ is the centralizer of $H$ in $K$ \cite{DKV}. The functions $F_{g,H}$ 
are well-studied, and they play a prominent
role in Harmonic analysis and convexity theory \cite{DKV,H,B,BB}.

Our main technical result establishes the following relation between the horospherical projection and the KKS symplectic structure of real hyperbolic orbits:
\begin{proposition}\label{pro:pro}
  Let $G$ be a connected, non-compact semisimple Lie group with finite center, let $G=KAN$ be any fixed Iwasawa decomposition and let
  $H\in \mathrm{Cl}(\aa^+)$. Then for any $g\in G$ the submanifold \[\mathrm{Ad}(K)_H^g\subset \mathrm{Ad}(G)_H\overset{i^{-1}}{\cong}T^*\mathrm{Ad}(K)_H\]
  is the graph of the exterior differential of  $-\Theta_{g^{-1}}^*F_{g,H}\in C^\infty(\mathrm{Ad}(K)_H)$,
  where $\Theta_g$ is the diffeomorphism  on $\mathrm{Ad}(K)_H$  induced by the action of $g$ on $\mathrm{Ad}(G)_H$ (which preserves the ruling).
 \end{proposition}

Proposition \ref{pro:pro} completes the description of the `horizontal' Lagrangians. The equality 
\[\w_{\mathrm{std}}(x)(Y,Z)=\w_{\mathrm{KKS}}(x)(Y,Z),\,\, \forall\, Y\in \cL_v,\,Z\in \cL_h\]
will follow from  computations analogous to those used to establish proposition \ref{pro:pro},  
 this providing a proof of the symplectomorphism \ref{eq:isom} which only appeals to basic symplectic geometry and Lie theory.

\section{Proof of the symplectomorphism \ref{eq:isom}}
In this section we fill in the details of the proof of the symplectomorphism \ref{eq:isom} sketched in the introduction.

Let us fix a Cartan decomposition $G=KP$ associated
to an involution $\theta$, and let $\mathfrak{k},\mathfrak{p}$ denote the respective Lie algebras. 
A choice maximal Abelian subalgebra $\aa\subset \pp$ and positive Weyl chamber $\mathfrak{a}^+\subset \aa$ (or root ordering)
gives rise to an Iwasawa decomposition $G=KAN$, with $\mathfrak{n}$ the Lie algebra of the nilpotent factor. 

In what follows   $X^g$ will denote the image of $X\in \mathfrak{g}$ by $\mathrm{Ad}(g)$.

We may pick without any loss of generality  $H\in \mathrm{Cl}(\aa^+)$ and consider the corresponding adjoint orbit $\mathrm{Ad}(G)_H$.
The orbit is identified
with the homogeneous space $G/Z(H)$, where $Z(H)$ denotes the centralizer of $H$. Under this identification $\mathrm{Ad}(K)_H$
is mapped to a submanifold canonically isomorphic to $K/Z_K(H)$, where $Z_K(H)=K\cap Z(H)$ is the centralizer of $H$ in $K$.
At the infinitesimal level,
the tangent of the real flag manifold  $T_H\mathrm{Ad}(K)_H$ is identified
with the quotient space $\mathfrak{k}/\mathfrak{z}_K(H)$, where $\mathfrak{z}_K(H)$ is the Lie algebra of $Z_K(H)$.

\begin{remark} When
using the quotient description of the coadjoint orbit and the real flag manifold, we
shall abuse notation and use representatives to denote the corresponding classes of points and tangent vectors, if
there is no risk of confusion.
\end{remark}

\subsection{The ruling and the identification $T^*\mathrm{Ad}(K)_H\overset{i}{\cong}\mathrm{Ad}(G)_H$.}
The contents we sketch in this subsection are rather standard. We refer the reader to \cite{ABB} for a  thorough exposition.

Let $\mathfrak{n}(H)$ be the sum of root subspaces associated to positive roots not vanishing on $H$. We have
the $\langle \cdot,-\theta\cdot \rangle$-orthogonal decomposition:
\[\mathfrak{g}=\theta\mathfrak{n}(H)\oplus \mathfrak{z}(H)\oplus \mathfrak{n}(H)\]

The affine subspace $H+\mathfrak{n}(H)$ is tangent to $\mathrm{Ad}(G)_H$ and complementary to $\mathrm{Ad}(K)_H$ at $H$.
Even more, the adjoint action of the subgroup
$N(H)$ integrating the nilpotent Lie algebra $\mathfrak{n}(H)$ maps $N(H)$ diffeomorphically
into $H+\mathfrak{n}(H)\subset \mathrm{Ad}(G)_H$.
This induces the well-known ruling of $\mathrm{Ad}(G)_H$ associated to the fixed Iwasawa decomposition.

As any ruled manifold, $\mathrm{Ad}(G)_H$ becomes an affine bundle. Since $\mathrm{Ad}(K)_H$ is transverse to the affine fibers,
the structure can be reduced to that of a vector bundle with zero section $\mathrm{Ad}(K)_H$. As to which vector bundle this is,
a vector tangent to the fiber over $H$ belongs to 
$\mathfrak{n}(H)$. The map $X\mapsto X+\theta X$ is a monomorphism from $\mathfrak{n}$ to $\mathfrak{k}$. Since
the image of $\mathfrak{n}(H)$ has 
trivial intersection with $\mathfrak{z}_K(H)$,  it is isomorphic to $\mathfrak{k}/\mathfrak{z}_K(H)\cong T_H\mathrm{Ad}(K)_H$. Therefore the pairing
$\langle\cdot,\cdot\rangle\colon \mathfrak{n}(H)\times \mathfrak{k}/\mathfrak{z}_K(H)\rightarrow \mathbb{R}$ --which is  well defined--
is also non-degenerate, and this provides the
canonical identification of the fiber at $H$ with $T^*_H\mathrm{Ad}(K)_H$. Since the Killing form and Lie bracket are $\mathrm{Ad}$-invariant, 
for any $k\in K$ we have the analogous statement for
\[\langle\cdot,\cdot\rangle\colon \mathfrak{n}(H^k)\times \mathfrak{k}/\mathfrak{z}_K(H^k)=\mathfrak{n}(H)^k\times 
\mathfrak{k}/\mathfrak{z}_K(H)^k\rightarrow \mathbb{R},\] this giving the identification:
\[i\colon T^*\mathrm{Ad}(K)_H\longrightarrow \mathrm{Ad}(G)_H.\]

\subsection{The symplectic forms $\w_{\mathrm{std}}$ and $\w_{\mathrm{KKS}}$.}
As remarked in the introduction, we shall omit the map $i$ in the notation, so we have $\w_{\mathrm{std}},\w_{\mathrm{KKS}}$
two symplectic forms on $\mathrm{Ad}(G)_H$ whose equality we want to check.

For the purpose of fixing the sign convention, we take the standard symplectic form of the cotangent bundle $\w_{\mathrm{std}}$ 
to be
$-d\lambda$, where $\lambda=\xi dx$ and $\xi,x$ are the momentum and position coordinates, respectively.

The tangent space at $H^g\in  \mathrm{Ad}(G)_H$ is spanned by vectors of the form $[X^g,H^g]$, $X\in \gg$. The formula
\[\w_{KKS}(H^g)([X^g,H^g],[Y^g,H^g])=\langle H,[X,Y]\rangle\]
is well defined on $\mathfrak{g}/\mathfrak{z}(H)$, and gives rise to an $\mathrm{Ad}(G)$-invariant symplectic form on the orbit \cite{Ki}.

As discussed in the introduction, to prove the equality $\w_{\mathrm{std}}(H^g)=\w_{\mathrm{KKS}}(H^g)$,
we shall start by finding complementary Lagrangian subspaces for both symplectic forms.

\subsection{The vertical Lagrangian subspaces}

At  $H^g$ we  define 
\[\cL_v(H^g):= H^g+\mathfrak{n}(H)^g,\] i.e. the tangent space to the ruling. Of course,
this space is Lagrangian for $\w_{\mathrm{std}}$. It is also Lagrangian for $\w_{\mathrm{KKS}}$ \cite{ABB}. We include
the proof of this fact to illustrate the kind of arguments we will use in our computations:

Two vectors in $\cL_v(H^g)$ 
are of the form $[X^g,H^g],[Y^g,H^g]$, where $X,Y\in \mathfrak{n}(H)$. Therefore
\[\w_{KKS}(H^g)([X^g,H^g],[Y^g,H^g])=\langle H^g,[X^g,Y^g]\rangle=\langle H,[X,Y]\rangle=0,\]
where the vanishing follows because $[X,Y]\in \mathfrak{n}(H)\subset \mathfrak{n}$ and  the subspaces $\mathfrak{a}$ and
are $\mathfrak{n}$ are orthogonal w.r.t. the Killing form (this following from the orthogonality w.r.t. the inner product 
$\langle\cdot,-\theta\cdot\rangle$ used in the Iwasawa decomposition).

\subsection{The horizontal Lagrangian subspaces} At $H^g$ we define
\[\cL_h(H^g):= T_{H^g}\mathrm{Ad}(K)_H^g.\] We shall prove that  $\mathrm{Ad}(K)_H^g$ is a Lagrangian submanifold for both symplectic
forms, so in particular $\cL_h(H^g)$ is a Lagrangian subspace.

The KKS symplectic form is $\mathrm{Ad}(G)$-invariant. Therefore it suffices to prove that $\mathrm{Ad}(K)_H$
is Lagrangian w.r.t $\w_{\mathrm{KKS}}$ to conclude that for all $g\in G$ the submanifold $\mathrm{Ad}(K)_H^g$
is Lagrangian w.r.t $\w_{\mathrm{KKS}}$. We reproduce the short proof that can be found in  \cite{ABB}:

At $H^k$ two vectors tangent to $\mathrm{Ad}(K)_H$ 
are of the form $[X^k,H^k],[Y^k,H^k]$, where $X,Y\in \mathfrak{k}$. Hence
\[\w_{KKS}(H^k)([X^k,H^k],[Y^k,H^k])=\langle H^k,[X^k,Y^k]\rangle=\langle H,[X,Y]\rangle=0,\]
where the vanishing follows from $[X,Y]\in \mathfrak{k}$ and the orthogonality of $\mathfrak{a}\subset \mathfrak{p}$ and
$\mathfrak{k}$ w.r.t. the Killing form .

To describe the behavior of $\mathrm{Ad}(K)_H^g$ w.r.t. to $\w_{\mathrm{std}}$, we require a formula for the projection map
$\mathrm{pr}\colon \mathrm{Ad}(G)_H\rightarrow \mathrm{Ad}(K)_H$ defined by the bundle structure.

The Iwasawa projections
\[K\colon G\rightarrow K,  \,\,A\colon G\rightarrow A,\,\,N\colon G\rightarrow N\]
are
characterized by $x\in K(x)AN$, $x\in KA(x)N$, $x\in KAN(x)$, respectively (note that the horospherical projection
cited in the Introduction is $H=\mathrm{log}A$).

The map  $K\colon G\rightarrow K$  descends to the bundle projection
\[\mathrm{pr}\colon  \mathrm{Ad}(G)_H\cong G/Z(H)\rightarrow \mathrm{Ad}(K)_H\cong K/Z_K(H)\]
associated to the ruling: write $g=K(g)A(g)N(g)$. The affine fiber $H+\mathfrak{n}(H)$ is preserved by $A(g)N(g)$, and, 
hence:
  \[H^g=H^{K(g)A(g)N(g)}\in (H+\mathfrak{n}(H))^{K(g)}.\]

In particular (see also  \cite{DKV}, section 3),  the automorphism on $K/Z_H(K)$ induced by the action of $g$ on $\mathrm{Ad}(G)_H$ is 
\begin{equation}\label{eq:action}\Theta_g\colon K/Z_K(H)\rightarrow K/Z_K(H),\,\, k\mapsto K(gk)
 \end{equation}

To understand the bundle projection infinitesimally we also need information on the differential of the Iwasawa projections.
This information can be found for $K$ and $H$ or $A$ in \cite{DKV} (for higher order derivatives as well; see also \cite{BB}).
The result for the three projections
is presented below; the proof is omitted since it is a straightforward application of the chain rule.

\begin{lemma}\label{lem:infpro}
For any $X\in \mathfrak{g}$ and $g\in G$ we have
\begin{equation}\label{eq:decomp2}
X^{AN(g)}=K(X,g)+A(X,g)^{A(g)}+N(X,g)^{AN(g)}
\end{equation}
written as sum of vectors in $\mathfrak{k},\mathfrak{a},\mathfrak{n}$,
where $K(X,g),A(X,g),N(X,g)$ stand for the left translation to the identity on $K,A,N$ of the vector field represented
by the curves $K(g\mathrm{exp}(tX)),A(g\mathrm{exp}(tX)),N(g\mathrm{exp}(tX))$, respectively, and $AN(g)$ denotes $A(g)N(g)$.
\end{lemma}

\begin{proof}[Proof of Proposition \ref{pro:pro}]

The submanifold $\mathrm{Ad}(K)_H^g$ is everywhere transverse to the fibers of the ruling, and hence corresponds to 
the graph of a 1-form $\alpha_{g,H}\in \Omega^1(\mathrm{Ad}(K)_H)$, which we compute
now: 

Given any $k\in K$ the point  $H^{gk}\in \mathrm{Ad}(K)_H^g$ projects
over  $H^{K(gk)}\in \mathrm{Ad}(K)_H$. The tangent space $T_{H^{K(gk)}}\mathrm{Ad}(K)_H\cong T_{K(gk)}K/Z_K(H)$ 
is spanned by vectors of the form  $L_{K(gk)^*}K(X,gk)$, where $K(X,gk)$ is the vector introduced in  Lemma
\ref{lem:infpro}.
By definition of $\a_{g,H}$ we have:
\[
\a_{g,H}(K(gk))(L_{K(gk)*}K(X,gk))=\langle (H^{gk}-H^{K(gk)})^{K(gk)^{-1}}, K(X,gk)\rangle.
\]
Because  $\mathfrak{k}$ and $\mathfrak{a}\subset \mathfrak{p}$ are $\langle \cdot,\cdot\rangle$-orthogonal, we deduce:
\[\langle (H^{gk}-H^{K(gk)})^{K(gk)^{-1}}, K(X,gk)\rangle=\langle H^{AN(gk)},K(X,gk)\rangle.
\]
By (\ref{eq:decomp2}) 
\[\langle H^{AN(gk)},K(X,gk)\rangle=\langle H^{AN(gk)},X^{AN(gk)}-A(X,gk)-N(X,gk)^{AN(gk)}\rangle.\]
Because  $\aa$ and $\mathfrak{n}$ are $\langle \cdot,\cdot\rangle$-orthogonal

\[\langle H^{AN(gk)},X^{AN(gk)}-A(X,gk)-N(X,gk)^{AN(gk)}\rangle=-\langle H^{AN(gk)},A(X,gk)\rangle.\]
Because  $H^{AN(gk)}-H\in \mathfrak{n}$,  we conclude:
\begin{equation}\label{eq:form1}
 \a_{g,H}(K(gk))(L_{K(gk)*}K(X,gk))=-\langle H,A(X,gk)\rangle.
\end{equation}

Now consider the function: 
\begin{eqnarray*}
 F_{g,H}\colon K &\longrightarrow &\R,\\
 k &\longmapsto &\langle H,H(gk)\rangle.
\end{eqnarray*}
By \cite{DKV}, Proposition 5.6, it is $Z_K(H)$-right invariant, and hence it defines a function on the real flag manifold
$K/Z_K(H)$, still denoted by $F_{g,H}$. This function is pulled back by the diffeomorphism
$\Theta_{g^{-1}}\colon K/Z_K(H)\rightarrow K/Z_K(H)$ to $\mathcal{F}_{g,H}\in C^\infty(K/Z_K(H))$. We have

\[\frac{d}{dt}\mathcal{F}_{g,H}(K(gk\exp(tX))_{\mid t=0}=
\frac{d}{dt}\langle H,H(gk\exp tX)\rangle_{\mid t=0} =\langle X^{AN(gk)},H\rangle,\]
where the last equality follows from \cite{DKV},
Corollary 5.2.

Applying  (\ref{eq:decomp2}) and the orthogonality relations we obtain 
\begin{equation}\label{eq:form2}
d \mathcal{F}ç_{g,H}(K(gk))(L_{K(gk)*}K(gk,X))=\langle H, A(X,gk)\rangle,
\end{equation}
and thus by equations (\ref{eq:form1}) and (\ref{eq:form2}) we conclude
\[\a_{g,H}=-d\mathcal{F}_{g,H},\]
as we wanted to prove.
\end{proof}

\subsection{The equality $\w_{\mathrm{KKS}}=\w_{\mathrm{std}}$.}

We just need to prove the equality at any point $H^g$  on pairs of vectors $[X^g,H^g], [Y^g,H^g]$,
where $X\in \mathfrak{k}$ and $Y\in \mathfrak{n}(H)$.

By definition of the KKS form: \[\w_{KKS}(H^g)([X^g,H^g],[Y^g,H^g])=\langle H,[X,Y]\rangle.\]

As for the standard form:
\[ \w_{\mathrm{std}}(H^g)([X^g,H^g],[Y^g,H^g])=\langle [Y^g,H^g]^{K(g)^{-1}},K(X,g)\rangle=\langle [Y^{AN(g)},H^{AN(g)}],K(X,g)\rangle.\]
By equation (\ref{eq:decomp2})
\[\langle [Y^{AN(g)},H^{AN(g)}],K(X,g)\rangle=\langle [Y,H],X\rangle -\langle [Y^{AN(g)},H^{AN(g)}], A(X,g)^{A(g)}+N(X,g)^{AN(g)}\rangle,\]
which equals $\langle [Y,H],X\rangle=\langle H,[X,Y]\rangle$ since in the second summand the first entry belongs to $\mathfrak{n}$ and the second to
$\mathfrak{a}+\mathfrak{n}$. 

\textsl{Acknowledgements:} I would like to thank the referee for pointing out related results which significantly improved the paper.

\end{document}